\documentclass{amsart}
\usepackage[latin1]{inputenc}
\usepackage[active]{srcltx}
\usepackage{fullpage}
\usepackage{amsfonts,enumerate}
\usepackage[mathscr]{euscript}
\usepackage{epsfig}
\usepackage{latexsym}
\usepackage[all]{xy}
\usepackage{amssymb}
\usepackage{amsthm}
\usepackage{amsmath}
\usepackage{amsxtra}
\usepackage{xcolor}
\usepackage[colorlinks]{hyperref}
\usepackage{fancyvrb}
\hypersetup{citecolor=blue}

.tfm
.tfm

\theoremstyle{plain}
\newtheorem{prop}{Proposition}[section]
\newtheorem{thm}[prop]{Theorem}

\newtheorem{lemma}[prop]{Lemma}

\newtheorem*{thm*}{Theorem}
\newtheorem*{lemma*}{Lemma}

\newtheorem*{prop*}{Proposition}

\theoremstyle{definition}
\newtheorem*{defi}{Definition}

\theoremstyle{remark}

\newtheorem{remark}{Remark}

\numberwithin{table}{section}

\DeclareMathOperator{\ab}{ab}

\DeclareMathOperator{\Cl}{Cl}
\DeclareMathOperator{\Gal}{Gal}

\DeclareMathOperator{\Ind}{Ind}

\newcommand{\F}{\mathbb F}

\newcommand{\Om}{{\mathscr{O}}}

\newcommand{\Disc}{\Delta}

\newcommand{\GL}{{\rm GL}}

\def\ZZ{\mathbb Z}

\def\QQ{\mathbb Q}

\def\CC{\mathbb C}
\def\ff{\mathfrak f}

\def\C{\mathcal C}

\def\<#1>{{\left\langle{#1}\right\rangle}}
\def\abs#1{{\left|{#1}\right|}}

\def\Z{{\mathbb Z}}             
\def\Q{{\mathbb Q}}             

\def\id#1{{\mathfrak{#1}}}      
\def\normid#1{{\norm{\id{#1}}}}
\DeclareMathOperator{\norm}{{\mathscr N}}

\newcommand{\lmfdbec}[3]{\href{http://www.lmfdb.org/EllipticCurve/Q/#1/#2/#3}{{\text{\rm#1-#2-#3}}}}

\let\kro\dkro

\begin{document}

\title{On the equation $x^2+dy^6=z^p$ for square-free $1 \le d \le 20$}

\author{Franco Golfieri}
\address{FAMAF-CIEM, Universidad Nacional de
	C\'ordoba. C.P:5000, C\'ordoba, Argentina.}
\email{franco.golfieri@mi.unc.edu.ar}
\thanks{}

\author{Ariel Pacetti}
\address{Center for Research and Development in Mathematics and Applications (CIDMA),
	Department of Mathematics, University of Aveiro, 3810-193 Aveiro, Portugal}
\email{apacetti@ua.pt}
\thanks{AP was partially supported by FonCyT BID-PICT 2018-02073 and by
the Portuguese Foundation for Science and Technology (FCT) within
project UIDB/04106/2020 (CIDMA). FG and LVT were supported by a CONICET grant.}

\author{Lucas Villagra Torcomian}
\address{FAMAF-CIEM, Universidad Nacional de
  C\'ordoba. C.P:5000, C\'ordoba, Argentina.}
\email{lucas.villagra@unc.edu.ar}
\thanks{}


\begin{abstract}
  The purpose of the present article is to show how the modular method
  together with different techniques can be used to prove
  non-existence of primitive non-trivial solutions of the equation
  $x^2+dy^6=z^p$ for square-free values $1 \le d \le 20$ following the
  approach of \cite{PT}. The main innovation is to make use of the
  symplectic argument over ramified extensions to discard solutions,
  together with a multi-Frey approach to deduce large image of Galois
  representations.
\end{abstract}

\maketitle

\section*{Introduction}

The study of Diophantine equations has been a major research area
since ancient times, and it got a lot of attention during the last
twenty years after Wiles' proof of Fermat's last theorem. Of particular interest is to understand the set of solutions to the generalized Fermat's equation
\[
  AX^p + BY^q = C Z^r,
\]
for $\frac{1}{p}+ \frac{1}{q} + \frac{1}{r} <1$. In the present article, we focus on the particular equation 
\begin{equation}
  \label{eq:mainequation}
  \C_d:x^2+dy^6=z^n.
\end{equation}
Specializing (\ref{eq:mainequation}) at $y=1$ gives the so called
Lebesgue--Nagell equation (see for example \cite{MR1259344} and the
references therein). Not so long ago, all solutions of the
Lebesgue-Nagell equation where obtained for $1 \le d \le 100$ (in
\cite{MR2196761}).  Going back to equation (\ref{eq:mainequation}), we
are mostly interested in the case $n > 3$ so that
$\frac{1}{2} + \frac{1}{6} + \frac{1}{n} < 1$. In the article
\cite{MR2966716} the authors studied equation
(\ref{eq:mainequation}) for $d=1$ (as part of their study of
generalized Fermat's equations).  Their main result is that for
$n\ge 3$ there are no \emph{non-trivial} and \emph{primitive} solutions. Let us
recall these two notions.

\begin{defi}
  A solution $(A,B,C)$ to (\ref{eq:mainequation}) is called
  \emph{primitive} if $\gcd(A,B,C)=1$ . A solution is called
  \emph{non-trivial} if $ABC \neq 0$.
\end{defi}

Contrary to Fermat's problem, equation (\ref{eq:mainequation}) is
non-homogeneous, hence it does not determine a projective variety, but
an affine surface. In particular, the set of primitive and
non-primitive solutions are very different. Moreover, there are many
non-primitive solutions!

\begin{prop*}[Granville] Let $p>3$ be a prime number. The equation
  \[
    \C_d:x^2 + dy^6 = z^p,
  \]
  has infinitely many non-primitive solutions.
\end{prop*}
\begin{proof}
  Suppose that $p \equiv 1 \pmod 6$. Let $u,v \in \ZZ$ arbitrary so
  that the value $r=u^2+dv^6$ does not equal $\pm 1$. Then the point
  $(ur^{(p-1)/2},v r^{(p-1)/6},r)$ lies in $\C_d$. If
  $p\equiv 5 \pmod 6$, then the point
  $(ur^{(5p-1)/2},vr^{(5p-1)/6},r^5)$ lies in $\C_d$.
\end{proof}
\begin{remark}
  The ABC conjecture implies that besides the solutions with
  $z = \pm 1$ (which are finite by Faltings' theorem), all other
  primitive ones are finite. In particular, this is the case when
  $d>0$.
\end{remark}

\begin{remark}
  Sometimes when $d<0$ a solution to Pell's equation provides a point
  in $\C_d$ with $z = \pm 1$. Such a point provides a non-trivial
  solution for all values of $p$, making the method to fail in
  general. This is the case for example for $(u,v,d) =
  (31,2,-15)$. For this reason, we restrict to
  positive values of $d$.
\end{remark}

In \cite{PT} a general strategy to prove non-existence of
non-trivial solutions was presented, including a few examples (for both positive and negative values of $d$, see also \cite{2103.06965}). Let us
recall how the method works: to a putative solution $(A,B,C)$ of
(\ref{eq:mainequation}) of prime exponent $p$, we attach the
$\Q$-curve
\begin{equation}
  \label{eq:frey-curve}
E_{A,B}:y^2 +6B\sqrt{-d} xy -4d (A+B^3\sqrt{-d})y = x^3.
\end{equation}
Its discriminant equals $-2^83^3d^4C^p(A+B^3\sqrt{-d})^2$.  Note that
the trivial solution $(\pm 1,0,1)$ corresponds to a rational curve
with complex multiplication (CM from now on) by
$\Z[\frac{1+\sqrt{-3}}{2}]$.

The main result of the aforementioned article consists of constructing a
character $\chi$ such that $\rho_{E}\otimes \chi$ descends to a
rational representation, which is modular by Serre's conjectures (say
attached to a newforms $f_{A,B}$). In particular, an explicit formula
for the level and the Nebentypus of the modular form was presented.

Note that the discriminant of the curve $E_{A,B}$ is ``almost'' a
perfect $p$-th power (except for the primes $2$, $3$ and the ones
dividing $d$). Then the general strategy is to apply
Ribet's lowering the level result (as in \cite{MR1104839}) to get a
congruence between $f_{A,B}$ and another newform whose level is not
divisible by primes dividing $C$, i.e. its level is only  divisible
by $2, 3$ and the primes dividing $d$. Ribet's result is only valid
when the image of the residual representation (modulo $p$) is
absolutely irreducible.

\vspace{5pt}

\noindent {\bf Problem 1:} Give an explicit constant $N_d$ such that
if $p>N_d$ then the Galois representation $\rho_{E_{A,B},p}$ has
absolutely irreducible residual image.

\vspace{5pt}

Then if $p>N_d$, there exists a newform $g$ in $S_2(\Gamma_0(N),\varepsilon)$
where $N$ is divisible only by  $2, 3$ and the primes dividing $d$ such that $f_{A,B} \equiv g \pmod p$. In \cite[Theorem 6.3]{PT} an explicit formula for $N$ is given.

The remaining obstruction for deciding whether
equation~(\ref{eq:mainequation}) has a non-trivial solution or not,
comes from the challenge to discard all newforms in the space
$S_2(\Gamma_0(N),\varepsilon)$ (as not coming from a possible
solution). 

\vspace{5pt}

\noindent {\bf Problem 2:} How can we discard the forms with complex multiplication appearing in such space?

\vspace{5pt}

Modular forms with CM are in general harder to discard, as for example
the trivial solutions correspond to such a curves! However, trivial
solutions are the unique ones corresponding to an elliptic curve with
CM under our assumption that $d$ is square-free when $d \neq 2$ (by
Lemma~\ref{lemma:CMsolutions}), hence we can just discard them while
searching for non-trivial solutions.

\vspace{5pt}
\noindent {\bf Problem 3:} How can we discard the non-CM forms?

\vspace{5pt}

The purpose of the present article is to show how different approaches
appearing in the literature (and some generalizations) can be combined
to answer Problem 3 and in particular prove non-existence (under
certain hypothesis) of solutions
to equation (\ref{eq:mainequation}) for most integral square-free
values of $d$ in the range $[1,20]$.

\vspace{2pt}

Regarding Problem $1$, let us recall the following result due to
Ellenberg (\cite{MR2075481} and \cite{MR2176151}).

\begin{thm*}(Ellenberg) Let $E/\Q(\sqrt{-d})$ (with $d>0$) be a
  $\Q$-curve satisfying that there exists a prime $p>3$ of multiplicative
  reduction for $E$. Then there exist an integer $N_d$ such that the
  projective image of the residual representation of $\rho_{E,p}$ is
  surjective for all primes $\id{p}$ of norm greater than $N_d$.
\label{thm:ellenberg}
\end{thm*}

In \cite{MR2075481} it is explained how to get an explicit bound for
$N_d$ (see also \cite[Theorem 6]{MR2561200}). Concretely, let $N$ be a any positive integer, and $\chi$ the character corresponding to $K/\QQ$ (of conductor $\ff$). Let $\mathcal{F}$ be a Petterson-orthogonal basis for $S_2(\Gamma_0(N))$. Define 
\begin{equation*}
\left(a_m,L_{\chi}\right)_{N}=\sum\limits_{f \in \mathcal{F}} a_m(f)L(f \otimes \chi,1).
\end{equation*}
If $M|N$, define $(a_m, L_{\chi})_N^M$ as the contribution from the old forms of level M. Then Ellenberg's result \cite{MR2075481} states that for any prime $p$ for which 
\begin{equation}
\label{eq:p-new}
\left(a_1, L_{\chi}\right)_{p^2}^{p-\text{new}} = (a_1,L_{\chi})_{p^2} -p(p^2-1)^{-1}(a_1-p^{-1}\chi(p)a_p,L_{\chi})_p
\end{equation}
is non-zero, the residual image is large (see
Section~\ref{section:ellenberg} for more details).  Note that
Ellenberg's result is strong enough to solve Problem 1 and Problem 2,
since modular forms with complex multiplication do not have surjective
projective image. A new problem appears while trying to apply
Ellenberg's result, namely.

\vspace{5pt}

\noindent {\bf Problem 4:} How can be assure that the curve $E_{A,B}$
has a prime $p>3$ of multiplicative reduction?

\vspace{2pt}

\noindent In Section~\ref{section:multi-frey} we present a solution to Problem
$4$, in particular giving a complete answer to Problem $2$.

The article is organized as follows: the first section contains four
different approaches to attack Problem 3, namely properties that
newforms related to real solutions must satisfy! Each approach
can be thought of as a ``real solution test'', so while trying to
discard a particular newform, we run the four different tests on
it. If the form fails to pass one test, we automatically discard
it. The four different tests include the so called Mazur's trick, the
study of the local inertial type, the application of the symplectic
argument (as explained in \cite{1607.01218}, with some improvements for
$\Q$-curves), and the property that all our curves contain a rational
point of order $3$ over $\Q(\sqrt{-d})$.

The second section contains one of the novelties of the present
article, namely to use a multi-Frey approach to give a complete answer
to Problem 4. As will be explained in
Section~\ref{section:multi-frey}, we succeed to run the multi-Frey
approach thanks to Cremona's list of elliptic curves of conductor up
to $500000$ (available at \cite{lmfdb}).

The last section contains applications of all the previous ones to
study solutions of (\ref{eq:mainequation}) for $1 \le d \le 20$
square-free. Regarding such examples, there are three cases that
cannot be handled for computational reasons (the level of the space of
modular forms needed to be computed is too large). They correspond to
the values $d=10$, $d=14$ and $d=17$. The symplectic argument used to
discard newforms only provides ``partial results'', namely, it only
allows to discard primes $p$ satisfying certain congruence
conditions. In particular, in some cases we can only prove that
equation~(\ref{eq:mainequation}) does not have any solution for an
explicit positive density of primes. This is the case for $d=7$,
$d=11$ and $d=15$. Full new results are obtained for $d=5$, $d=13$ and
$d=19$.

The computations of the present article were done using Pari/GP
(\cite{PARI2}), Magma (\cite{MR1484478}) and Sage
(\cite{sagemath}). The Pari script implemented to compute Ellenberg's
bound, the one implemented (in Magma) to apply Mazur's trick and the
computations made in Magma (as well as their outputs) to certify the
present results can be found in the web page
\url{http://sweet.ua.pt/apacetti/research.html}.

\subsection*{Acknowledgments} We would like to thank Professor John
Cremona for explaining us how to compute building blocks for newforms
of degree two via the computation of the period matrix (as illustrated
in the examples section for $d=11$). We would like to thank Professor
Mike Bennett for fruitful discussions on how to use the multi-Frey
method and tables of rational elliptic curves to solve Problem 4. At
last (but not least) we thank Professor Samir Siksek for explaining us
how to use the existence of a three torsion point to discard elliptic
curves (as detailed in Section~\ref{section:3torsion}).

\section{Some tests to eliminate forms}
To easy notation, during the whole article, we will denote by $K$ the
field $\Q(\sqrt{-d})$.

\subsection{Mazur's trick}
\label{section:marzurtrick} Let $f \in S_2(\Gamma_0(N),\varepsilon)$
be a newform, where $N$ is the level obtained in \cite{PT} after
applying the lowering the level result of Ribet to a putative solution
$(A,B,C)$. Let $p$ and $\ell$ be different prime numbers such that
$\ell$ does not ramifying in $K/\Q$. Let $\id{l}$ be a prime in $K$
dividing $\ell$. Let $f^{\text{BC}}$ denote the base change of $f$ to
$K$. Recall that the $q$-expansion of $f^{\text{BC}}$ is given by the
following rule:
\begin{equation}
  \label{eq:BC}
  a_{\id{l}}(f^{\text{BC}}) =
  \begin{cases}
    a_\ell(f) & \text{ if } \normid{l} = \ell,\\
    a_\ell(f)^2-2\ell \varepsilon(\ell) & \text{ if }\normid{l} = \ell^2.
\end{cases}
\end{equation}
The idea behind Mazur's trick is to
study equation~(\ref{eq:mainequation}) modulo $\ell$, to get
information on $a_{\id{l}}(E_{A,B})$. Let
\[
  S_\ell = \{(\tilde{A},\tilde{B},\tilde{C}) \in \F_{\normid{(l)}} \; : \;
  \tilde{A}^2+d\tilde{B}^6=\tilde{C}^{\ell}\},
  \]
  discarding the trivial solution $(0,0,0)$, where $\normid{(l)}$
  stands for the norm of the ideal $\id{l}$. For each point
  $(\tilde{A},\tilde{B},\tilde{C}) \in S_\ell$, consider the curve
  $E_{\tilde{A},\tilde{B}}$ over $\F_{\normid{(l)}}$. Then, either:
  \begin{itemize}
  \item The curve $E_{\tilde{A},\tilde{B}}$ is non-singular, in which case if
    $(A,B,C)$ is an integral solution reducing to
    $(\tilde{A},\tilde{B},\tilde{C})$, we must have that
    $a_{\id{l}}(E_{A,B}) = a_{\id{l}}(E_{\tilde{A},\tilde{B}})$ and
      furthermore
      \[
\chi(\id{l}) a_{\id{l}}(E_{\tilde{A},\tilde{B}}) \equiv a_{\id{l}}(f^{\text{BC}}) \pmod p,
\]
or
\item The curve $E_{A,B}$ has bad reduction at $\id{l}$ in which case we are in the lowering the level hypothesis, and
  \[
a_\ell(f)^2 \equiv \varepsilon^{-1}(\ell) (\ell+1)^2 \pmod p.
    \]
  \end{itemize}
  In both cases, each side of the congruence can be computed, and if
  they happen to be different, we get a finite list of candidates for
  the prime $p$ dividing their difference.
   This method is very powerful, and running it for a few different
  values of $\ell$ allows to discard all newforms whose coefficient
  field (of the base change to $\Q(\sqrt{-d})$) does not match 
  $\Q(\chi)$. Sometimes, it does discard all possible forms (providing
  non-existence of solutions of (\ref{eq:mainequation})), however in
  many instances, there are newforms which systematically pass the test.

  \subsection{The local type} Let $K_\lambda$ be a finite extension of
  $\Q_\ell$. The local-Langlands correspondence gives a bijection
  between the set of automorphic representations of $\GL_2(K_\lambda)$
  and representations of the Weil-Deligne group of $K_\lambda$ (we
  will follow Carayol's normalization of \cite{MR870690}). Recall that
  a Weil-Deligne representation consists of a two dimensional complex
  representation $\rho$ of the Weil group $W(K_\lambda)$ together with
  a monodromy operator $N$. Let $\omega_1: W(\Q_\ell) \to \CC^\times$
  be the unramified character sending Frobenius to $\|\ell\|_\ell$
  (and use the same notation for its restriction to
  $W(K_\lambda)$). For $\ell \neq 2$, there are three different local
  types for a Weil-Deligne representation $\rho$ with trivial
  Nebentypus:

\begin{enumerate}
\item {\bf Principal Series}: the endomorphism $N = 0$ and 
$\rho = \chi \oplus \chi^{-1}\omega_1^{1-k}$
for some quasi-character $\chi:W(\QQ_p)^{\ab} \mapsto \CC^\times$.
\item{\bf Steinberg or Special Representation}: The endomorphism $N$
  is given by the matrix
  $\left(\begin{smallmatrix}0 & 1\\ 0 & 0 \end{smallmatrix} \right)$
  and the representation $\rho$ equals
  $\omega_1^r \left(\begin{smallmatrix}\chi \omega_1 & 0\\
      0 & \chi\end{smallmatrix} \right)$ for some quasi-character
  $\chi:W(K_\lambda) \mapsto \CC^\times$.
\item{\bf Supercuspidal Representation}: the endomorphism $N=0$ and
  $\rho=\Ind_{W(E)}^{W(K_\lambda)}\varkappa$ where $E$ is a
  quadratic extension of $K_\lambda$, and
  $\varkappa:W(E)^{\ab} \to \CC^\times$ is a quasi-character
  which does not factor through the norm map with a quasi-character of
  $W(K_\lambda)^{\ab}$.
\end{enumerate}

\begin{remark}
  There is a big difference between a principal series representation
  and a supercuspidal one, as the first one is reducible (and
  decomposable) while the second one is not.
\end{remark}

The local inertial type of an automorphic representation of
$\GL_2(\Q_\ell)$ is the isomorphism class of its restriction to the
inertia subgroup (which is related to the restriction of its
Weil-Deligne counterpart).

\begin{prop}
  \label{prop:localtypeinvariance}
  Let $\rho:\Gal_{K} \to \GL_2(\overline{\Q_p})$ be a modular Galois
  representation, and let $\id{l}$ be a prime of $K$ whose residual
  characteristic is prime to $p$. If the local type of $\rho$ at
  $\id{l}$ is principal series (respectively supercuspidal) given by a
  character $\chi$ (respectively $\varkappa$) of order $n$ prime to $p$
  then its reduction is also of principal series type (respectively
  supercuspidal) given by a character of order $n$.
\end{prop}

\begin{proof} Let $K_\lambda$ denote the completion of $K$ at the
  ideal $\id{l}$.  By local class field theory, the character $\chi$
  appearing in the restriction $\rho|D_{\id{l}}$ to a decomposition
  group at $\id{l}$ comes from a Hecke character
  $\chi:W(K_\lambda) \to \overline{\Q_p}$. The kernel of the reduction
  map $\overline{\ZZ_p} \to \overline{\F_p}$ is a pro-p-group. Since
  the order of $\chi$ is prime to $p$, the reduced character
  $\tilde{\chi}$ has the same order as $\chi$, proving the first
  statement. In the supercuspidal case, the residual character
  $\tilde{\varkappa}$ has the same order as $\varkappa$, hence we are only led
  to prove it does not factor through the norm map (so the residual
  representation is also irreducible). If
  $\tilde{\varkappa} = \norm \circ \tilde{\theta}$, for some character
  $\tilde{\theta}: W(E) \to \overline{\F_p}$, let
  $\theta: W(E) \to \overline{\Q_p}$ be the Teichmuller lift of
  $\overline{\theta}$. Then $\kappa = \norm \circ \theta$ (as both
  characters have order prime to $p$ and their reductions coincide).
\end{proof}

\begin{remark}
  \label{rem:inertiatype}
  The representations $\rho$ appearing in the present article come
  from $\Q$-curves (whose modularity is known due to Serre's
  conjectures). Since the coefficient field of an elliptic
  curve is the rational one, any character (in both the principal
  series or the supercuspidal type) appearing at a prime of bad
  reduction has order $n$ with $\varphi(n) \le 2$. In particular,
  $n \in \{1,2,3,4, 6\}$. Then for $p>3$, the local type is preserved
  by congruences.
\end{remark}

The $\ell$-inertial type of the $p$-adic representation $\rho$ is
the isomorphism class of the restriction of $\rho$ to the
decomposition group of $\ell$.
%
There is an algorithm for given a rational newform and a prime
dividing the level, compute its local type (based on \cite{MR2869056})
which is implemented both in Sage and Magma. The problem is that in
most instances, the space $S_2(\Gamma_0(N),\varepsilon)$ has very large
dimension, making the computation unfeasible (from a computational
point of view).

\subsection{The symplectic argument}
\label{section:symplectic} The symplectic argument is a powerful tool to
study congruences between elliptic curves. The idea is to consider not
only when two elliptic curves have isomorphic residual
representations, but add information on how the isomorphism relates
their Weil's pairing. Its first version (and application)
appeared in \cite{MR1166121} (see \cite{1607.01218} for the different
historical applications and the latest results). Let us state two
instances of the symplectic argument that will be used in the present
article.

\begin{thm}
  \label{thm:classicalsymplectic}
  Let $\ell \neq p$ be primes with $p \ge 3$. Let $E, E'$ be elliptic
  curves over $\Q_\ell$ with multiplicative reduction. Suppose that
  $E[p] \simeq E'[p]$ as $\Gal_{\Q_\ell}$-modules. Assume furthermore
  that $p \nmid v_\ell(\Disc(E))$. Then $p \nmid v_\ell(\Disc(E'))$
  and furthermore, $E[p]$ and $E'[p]$ are symplectically isomorphic
  $\Leftrightarrow$ $\kro{v_\ell(\Disc(E))/v_\ell(\Disc(E'))}{p}=1$.

  Furthermore, both curves cannot be symplectically and anti-symplectically isomorphic.
\end{thm}

\begin{proof}
  See \cite{MR1166121} (and \cite[Theorem 13]{1607.01218}).
\end{proof}

We also need a version for ramified extensions. Recall that if $K$ is
a local field and $E/K$ is an elliptic curve with potentially good
reduction, the \emph{defect} of $E$ is the degree of the minimal
extension over $K^{\text{ur}}$ where $E$ attains good reduction.

\begin{thm}
  \label{thm:ramifiedsymplectic}
  Let $\ell \equiv 2 \pmod 3$ be a prime number, $K$ be a quadratic
  ramified extension of $\Q_\ell$, and let $p \neq \ell$ be a prime
  with $p \ge 5$. Let $E, E'$ be two elliptic curves over $K$ with a
  point of order $3$ on $K_\id{l}$ (where $\id{l}$ is a prime in $K$ dividing $\ell$), with potentially good reduction of defect
  $3$. Set $r=0$ if $v(\Disc(E)) \equiv v(\Disc(E')) \pmod 3$ and
  $r=1$ otherwise. Suppose that $E[p]$ and $E'[p]$ are isomorphic as
  $\Gal_{K}$-modules. Then $E[p]$ and $E'[p]$ are symplectically
  isomorphic $\Leftrightarrow$
  $\kro{\ell}{p}^r=1$.

    Furthermore, both curves cannot be symplectically and anti-symplectically isomorphic.
\end{thm}

\begin{proof}
  The proof mimics Case 1 in page 79 of \cite{1607.01218}. A key property
  of our hypothesis is that the fact that $\ell \equiv 2 \pmod 3$ and
  that $K/\Q_\ell$ is ramified, implies that there are no cubic roots
  of unity in $K$ (besides the trivial one), hence the extension of
  $K$ obtained by adding the $p$-torsion points is not abelian (by
  Corollary 5 of \cite{1607.01218}).

  The fact that there are no cubic roots of unity in $K$, implies that
  raising to the third power is a bijection on the residue fields,
  hence Lemma 15 of loc. cit. holds without any change. In particular,
  the same matrix representation of Frobenius and inertia given in
  Lemma 25 holds. With these ingredients at hand, it is clear that the
  proof of Case 1 in page 79 of \cite{1607.01218} holds mutatis mutandis.
\end{proof}



\subsection{Using the $3$-torsion information}
\label{section:3torsion}
Consider the following general situation: Let $E$ and $\tilde{E}$ be
rational elliptic curves. Let $N$ be a prime number, and suppose that
$E$ has an $N$-torsion point and that for a large prime $p$ there is
an isomorphism of Galois modules $E[p] \simeq \tilde{E}[p]$. How does
the $N$-torsion point of $E$ affect the curve $\tilde{E}$?  We thank
Professor Samir Siksek for explaining us the following result.

\begin{thm}
  There exists an
  explicit constant $M$ (depending only on the conductor of $\tilde{E}$) such that if
  $p \ge M$, then the elliptic curve $\tilde{E}$ has a
  $N$-rational point, or it is $N$-isogenous to a curve with an
  $N$-rational torsion point.
\label{thm:Samir}
\end{thm}

\begin{proof}
  Since the curve $E$ has an $N$-rational point, the semisimplification
  of its residual representation $\overline{\rho_{E,N}}^{\text{ss}}$
  is isomorphic to $1 \oplus \chi_N$ (the cyclotomic character). In
  particular, $a_q(E) \equiv 1 + q \pmod N$ for all primes $q \neq N$
  of good reduction. On the other hand, since
  $E[p] \simeq \tilde{E}[p]$, $a_q(E) \equiv a_q(\tilde{E}) \pmod
  p$. By Hasse's bound, the numbers $\abs{a_q(E)} \le 2\sqrt{q}$, so
  if $p > 4\sqrt{q}$, $a_q(E) = a_q(\tilde{E})$. Since $\tilde{E}$ is
  modular (by \cite{MR1839918}), it is attached to a newform in
  $S_2(\Gamma_0(\id{N}))$ ($\id{N}$ being its conductor). Let $M$ be the Sturm
  bound for modular forms in $S_2(\Gamma_0(\id{N}))$ (as in \cite[Theorem
  9.21]{MR2289048}), so two forms in $S_2(\Gamma_0(\id{N}))$ are congruent
  modulo $p$ if and only if they eigenvalues are congruent modulo
  $p$ for all primes up to $M$.

  If $p \ge 4\sqrt{M}$, the newform attached to $\tilde{E}$ is
  congruent modulo $N$ to an Eisenstein series. In particular,
  $a_q(\tilde{E}) \equiv q+1 \pmod N$ for all primes $q \neq N$ of
  good reduction. This implies that
  $\overline{\rho_{\tilde{E},N}}^{\text{ss}} \simeq 1 \oplus \chi_N$
  (by the Brauer-Nesbitt theorem), hence either
  \[
    \overline{\rho_{\tilde{E},N}} \simeq \left( \begin{smallmatrix} 1
        & * \\ 0 & \chi_N\end{smallmatrix} \right) \qquad \text{ or }
    \qquad \overline{\rho_{\tilde{E},N}} \simeq
    \left( \begin{smallmatrix} \chi_N & * \\ 0 & 1\end{smallmatrix}
    \right).
    \]
    In the first case, $\tilde{E}$ has a rational point of order $N$,
    while in the second case, it has a rational subgroup of order $N$,
    whose quotient is an (isogenous) curve with a point of order $N$.
\end{proof}

\begin{remark}
  We will apply the previous result to elliptic curves over an
  imaginary quadratic field with a point of order $3$ and try to
  discard possible curves $\tilde{E}$ which do not have a $3$-torsion
  point, nor a $3$-isogenous curve with a point of order $3$. Although
  effective Sturm bounds for Bianchi modular forms are hard to get, we
  could use the fact that our curves are $\Q$-curves, hence
  are related to rational newforms to get one. However, from a computational
  point of view, it is more effective to overcome this issue via
  computing (if it exists) the first prime $\id{q}$ such that
  $a_{\id{q}}(\tilde{E}) \not \equiv 1 + \normid{(q)} \pmod 3$. If
  such a prime exists, then $E$ and $\tilde{E}$ cannot have Galois
  representations which are congruent modulo $p$ for
  $p > 4\sqrt{\normid{q}}$.
\label{remark:siksek}  
\end{remark}

\section{The multi-Frey approach}
\label{section:multi-frey} The idea of the multi-Frey technique (as
developed by S. Siksek in \cite{MR2215137} and \cite{MR2414954}) is to
attach not one elliptic curve to a putative solution, but many of
them. In this way, even if Mazur's trick fails for a particular value
in $S_\ell$ for one curve, it may work for another one. This is
precisely the case while dealing with equation (\ref{eq:mainequation})
for $d=1$ (see Section $5$ of \cite{MR2966716}). Unfortunately, using
the same method for other values of $d$, we did not succeed to discard
any solution that passed Mazur's trick for our original
$\Q$-curve. However, the existence of a new curve attached to a
solution turns out to be very useful to deal with Problem $4$.
More concretely, to a putative solution $(A,B,C)$ of
(\ref{eq:mainequation}), attach the rational elliptic curve
\begin{equation}
  \label{eq:multifreycurve}
  \widetilde{E_{A,B}} \; \; : \; \; Y^2=X^3+3dB^2X+2dA.
\end{equation}
Its discriminant equals
$\Delta(\widetilde{E_{A,B}}) = -1728\cdot d^2\cdot C^p$ and its
$j$-invariant $\frac{1728\cdot d\cdot B^6}{C^p}$. Since $(A,B,C)$ is
primitive, $\gcd(d,C)=1$, hence it has multiplicative reduction at all
primes dividing $C$ and additive reduction at all primes $\ell > 3$
dividing $d$ if $\ell^6 \nmid d$.

\subsubsection{Application to Problem 4.} Consider the particular case
of equation~(\ref{eq:mainequation}) when $n$ is prime and $z$ is only
divisible by the primes $2$ or $3$, i.e. the equation
\begin{equation}
  \label{eq:multprime}
C_p:  x^2+dy^6=(2^a3^b)^p.
\end{equation}
To a putative solution $(A,B,C)$ we attach the rational elliptic curve
as in (\ref{eq:multifreycurve}), with discriminant $-2^{ap-6} 3^{bp+3} d^2$ and $c$-invariants as follows:
\begin{equation}\label{eq:c-invariants}
	c_4=-2^43^2dy^2, \quad c_6=-2^63^3dx.
\end{equation}
We thank Professor Mike Bennett for the following local description of the curve.
\begin{prop}
  \label{prop:multifreyconductor}
  The model is minimal at all primes $\ell \geq 3$. Furthermore,
  suppose that $d$ is $6$-th power-free. Then
  \begin{itemize}
  \item The curve has additive reduction at each prime $\ell \mid d$, $\ell > 3$.
    
  \item If $3 \nmid d$, the model is minimal at $3$ and
    $v_3(N_E) \in \{ 2, 3 \}$, with the latter case possible only if
    $b=0$ or $(b,p)=(1,2)$.
  \item If $v_3(d) \in \{ 1, 2, 4, 5 \}$, then $E$ is minimal at $3$ and
$v_3(N_E)=5$.
\item If $v_3(d) =3$, $E$ is minimal at $3$ and $v_3(N_E) \in \{ 2, 3 \}$.
\item If $2 \nmid d$, and our model is minimal at $\ell=2$, we have
  $v_2(N_E) \in \{ 2, 3, 4, 5, 6 \}$. If it is non-minimal (whereby
  necessarily $ap \geq 6$) a minimal model has $c$-invariants
  $c_4=-3^2 dB^2$, $c_6=-3^3 dA$, minimal discriminant
  $\Delta_E = -2^{ap-6} 3^{bp+3} d^2$, and $v_2(N_E) \in \{ 0, 1 \}$.
\item If $2 \mid d$, $E$ is minimal at $2$ if $8 \nmid d$. Furthermore,
  \begin{enumerate}
  \item If $v_2(d) =1$, $v_2(N_E) \in \{ 2, 3, 4, 7 \}$.
    
  \item If $v_2(d) =2$, $v_2(N_E) =6$.
  \end{enumerate}
\item If $\nu_2(d) =3$, then either $E$ is minimal at $2$ and $v_2(N_E)  \in \{ 4, 5 \}$,
or $E$ is non-minimal at $2$, $2 \mid B$, and a minimal model has $2 \nmid N_E$.
\item If $v_2(d) =4$, then $E$ is minimal at $2$ and $v_2(N_E)  =6$.
\item If $v_2(d) =5$, then $E$ is not minimal at $2$ and a minimal model has 
$v_2(N_E)  \in \{ 2, 3, 4 \}$.
\end{itemize}
\end{prop}

\begin{proof}
Follows from a straightforward computation using \cite{MR1225948}.
\end{proof}

For a particular value of $d$, we can search Cremona's tables
(\cite{MR1628193}) of elliptic curves, as available on the LMFDB
\cite{lmfdb} and from their discriminant, we get a finite list of
possibilities for $a$ and $b$.

\begin{thm}
  \label{thm:Bennet}
  Let $2 \le d \le 19$ be a square-free positive integer. Then for
  each value of $d$, the value of $p$ in (\ref{eq:multprime}) is
  bounded by the values of Table~\ref{table:Bennett}.
   \begin{table}[h]
  	\begin{tabular}{|l|c||l|c||l|c||l|c|}
  		\hline
  		$d$ & Bound &   $d$ & Bound &   $d$ & Bound &  $d$ & Bound \\
  		\hline
  		$2$ & $3$  & $3$ & $2$ & $5$ &  $2$ & $6$ & $-$\\
  		\hline
  		$7$ & $7$& $10$ &  $-$& $11$ & $5$& $13$ & $-$\\
  		\hline
  		$14$ &  $-$ & $15$ & $3$& $17$ &  $2$& $19$ & $-$\\
  		\hline
  	\end{tabular}
  	\caption{Bound for curves having a multiplicative reduction prime. 	The ``$-$'' symbol means that $C$ cannot be supported in $\{2,3\}$.}\label{table:Bennett}
  \end{table}
\end{thm}

\begin{proof}
  Let us explain the algorithm in one example: let $d=2$. Then
  Proposition~\ref{prop:multifreyconductor} gives that
  $N_E = 2^\alpha 3^\beta$ with $\alpha \in \{2,3,4,7\}$ and
  $\beta \in \{2,3\}$. Now we do a search for all the elliptic curves
  of conductor $2^\alpha 3^\beta$ which satisfy that its invariants
  ($c_4$, $c_6$, $\Delta$) are compatible with the invariants of our
  Frey curve, for instance, that are negative numbers (see equation
  ~(\ref{eq:c-invariants})) and that $288\mid c_4$ and $3456\mid
  c_6$. There is a single curve matching these three requirements
  labeled \lmfdbec{1152}{r}{2} in the LMFDB. Formula
  ~(\ref{eq:c-invariants}) allows us to recover the value of $(x,y)$,
  namely $(5,1)$. Then, factorizing $x^2+2y^6$ gives that $p$ must be
  $3$.
  
  A similar computation was done to each other value of $d$, giving
  the finite list of possibilities for the value of $p$ presented in
  the table. We want to stress that the described method fails for
  $d=19$ as it involves computing elliptic curves of conductor
  $623808$. There are two possibilities to try to follow the same
  approach: one is to compute such set of elliptic curves using
  Cremona's algorithm (this was done for us by Professor John Cremona),
  while the other is to use the tables of elliptic curves with bad reduction
  at a small set of primes as described in \cite{1605.06079} (whose
  database is available at
  \url{https://bmatschke.github.io/solving-classical-diophantine-equations/}). However,
  in this particular case, there is no need to perform such a
  computation. The reason is that since $2$ is inert, $2\nmid C$ (by
  \cite[Lemma 5.3]{PT}) and the same happens for the prime $3$ (by
  \cite[Lemma 5.4]{PT}).
\end{proof}

\begin{remark}
  The last argument actually proves that there is no solution for all
  values of $d$ which satisfy $d \not \equiv 7 \pmod 8$ and
  $d \not \equiv 2 \pmod 3$.
\end{remark}

\section{Application to (\ref{eq:mainequation}) with $1 \le d \le 20$ and square-free}\label{section:aplplications}

In this section, we apply the previous procedures to study solutions
of (\ref{eq:mainequation}) for $1 \le d \le 20$, square-free
(and show what are its limitations). The case $d=1$ was considered in
\cite{MR2966716}, the case $d=2$ in \cite{PT} and the case $d=3$ in
\cite{Angelos}, so we restrict to values $d \ge 5$.

To discard all CM newforms, we need to understand which solutions provide
elliptic curves with CM.
\begin{lemma}\label{lemma:CMsolutions}
Let $d$ be square-free and $(A,B,C)$ be a solution of
~(\ref{eq:mainequation}). Then the curve $E_{A,B}$ has CM if and only if
the solution is trivial or	$(A,B,C,d,p) =(\pm 5,\pm 1,3,2,3)$, with complex multiplication by $\Z[\sqrt{-2}]$.
\end{lemma}
\begin{proof}
  The $j$-invariant
  $j(E_{A,B})=2^43^3\sqrt{-d}B^3\cdot\frac{(4A-5B^3\sqrt{-d})^3}{C^p(A+B^3\sqrt{-d})^2}\in\Q({\sqrt{-d}})$. If
  $E$ has CM by an order $\Om \subset K$ its $j$-invariant lies in the
  Hilbert class field of $\Om$ (an extension of $K$). The extension
  $\Q(j(\Om))$ is quadratic imaginary if and only if it is the
  rational field. The imaginary part of $j(E_{A,B})$ factors as
  \begin{equation}\label{eq:j-invariant}
    864AB^3\cdot\frac{(2A^2-25dB^6)(16A^2-11dB^6)}{C^{3p}}.
  \end{equation}
  It vanishes if and only if $A=0$ (providing a non-primitive
  solutions), $B=0$ (trivial solutions) or one of the last two terms
  vanishes. Since $d$ is square-free and our solution is primitive,
  this can only occur when $d=2$, $A=\pm 5$ and $B= \pm 1$,
  corresponding to the point $(\pm 5, \pm 1, 3)$ for $p=3$.
\end{proof}

\subsection{Explicit Ellenberg's bounds}
\label{section:ellenberg}
%
%
%

\begin{prop}
\label{thm:ellbounds}
The bound $N_d$ in Ellenberg's result can be taken as: $N_5= 1033$,
$N_6 = 1289$, $N_7=337$, $N_{11}=557$, $N_{13}=3491$, $N_{15}=743$ and $N_{19}=1031$.
\end{prop}

\begin{proof}
  Let $\chi$ the quadratic character attached to $K/\QQ$ and $q$ its
  conductor. Let $N=p^2$, with $N \geq 400$, and let
  $\sigma=q^2/2\pi $. Then Ellenberg \cite[Theorem 1]{MR2176151} give
  us the following formula:
\begin{equation}
\label{eq:Ell05Th1}
(a_1,L_{\chi})_{p^2} = 4\pi e^{-2\pi/\sigma N \log(N)} - E^{(3)} +E_3 -E_2-E_1+ (a_1, B(\sigma N \log(N))),
\end{equation}
where bounds for $(a_1, B(\sigma N \log(N))), E_1, E_2, E_3$ and
$E^{(3)}$ are specified in \cite[Theorem 1]{MR2176151}. Denote the
previous bounds as $\text{bound1(p,q)}$, $E_1(p,q)$, $E_2(p,q)$,
$E_3(p,q)$ and $E_5(p,q)$ respectively.  In order to obtain a simpler
formula for the bound of $E^{(3)}$ (following Ellenberg's notation)
one splits the sum depending on whether $c \leq p^4$ or $c>p^4$. In
this way (as in \cite[Lemma 8]{MR2561200}), we obtain:
\begin{equation}
  \label{eq:E(3)}
  |E^{(3)}| \leq 16 \pi^3 \left( \frac{12 \phi(q)\log^2(p)}{\pi p^2} + \frac{q^2\log(p^2)}{4\pi p} \left( \zeta^2\left(\frac{3}{2}\right) - \sum\limits_{k=1}^{p^2} \frac{\tau(k)}{k^{3/2}} \right) \right),
\end{equation}
where $\phi$ is Euler's function and $\tau(k)=\sum_{d|k}1$.
Let $E_4(p,q)$ denote the right hand side. Let
$F(p,q) := 4\pi e^{-2\pi^2/p^2q\log(p)} - E_4(p,q) - E_3(p,q) -
E_2(p,q) -E_1(p,q) - \text{bound1}(p,q)$.  Then from Equations
\eqref{eq:p-new}, \eqref{eq:Ell05Th1}, \eqref{eq:E(3)} we have:
\begin{equation}
  \label{eq:p-newineq}
  \left(a_1, L_{\chi}\right)_{p^2}^{p-\text{new}} \geq  F(p,q) - \frac{1}{p^2-1}(a_p,L_{\chi})_p - \frac{p}{p^2-1}(a_1,L_{\chi})_p.
\end{equation}
In order to bound $(a_p,L_{\chi})_p$ and $(a_1,L_{\chi})_p$ we use
Ellenberg's bound \cite[Theorem 3.13]{MR2075481}, obtaining a function
$F_2(p,q,m)$ such that $(a_m,L_{\chi})_p \leq F_2(p,q,m)$. Therefore,
using this bound and the Inequality \eqref{eq:p-newineq} we have that:
\begin{equation}
  \label{eq:p-newineq2}
  \left(a_1, L_{\chi}\right)_{p^2}^{p-\text{new}} \geq  F(p,q) - \frac{1}{p^2-1}F_2(p,q,p) - \frac{p}{p^2-1}F_2(p,q,1).
\end{equation}
Note that, since $4\pi e^{-2\pi^2/p^2q\log(p)}$, $- E_4(p,q)$,
$- E_3(p,q)$, $- E_2(p,q)$, $-E_1(p,q)$ and $- \text{bound1}(p,q)$ are
increasing functions (in $p$), the function $F(p,q)$ is
increasing. One can also deduce that $-\frac{1}{p^2-1}F_2(p,q,p)$ and
$-\frac{p}{p^2-1}F_2(p,q,1)$ are also increasing functions, hence
$\left(a_1, L_{\chi}\right)_{p^2}^{p-\text{new}}$ is increasing. Then,
in order to assure that
$\left(a_1, L_{\chi}\right)_{p^2}^{p-\text{new}}$ is non-zero, it is
enough to get the first prime $p$ where the right hand side of
(\ref{eq:p-newineq2}) is positive. We implemented this code, following
the above notation, in Pari/GP, the resulting file being labeled
``ellenberg.gp'' (available at \url{http://sweet.ua.pt/apacetti/research}).

The values $d=5,6,7,11,13,15,19$ correspond to characters of
conductors $20,24,7,44,52,15,19$ respectively. Running our script
(EllenbergBound(p,q)), we obtain that the first prime $p$ for which
the bound is positive for
$d=5,6,7,11,13,15,19$ equals $1033$, $1289$, $337$, $557$, $3491$,
$743$ and $1031$ respectively.
\end{proof}

\begin{remark}
  \label{thm:ellboundsimp}
  In some instances the previous bound can be improved for each value
  of $d$ by a finite computation (see \cite[Proposition
  3.9]{MR2075481}) via computing for each $p < N_d$ whether there is a
  newform $f$ satisfying:
  \begin{itemize}
  \item $f \in S_2(\Gamma_0(dp^2))$ with $w_pf=f$ and $w_df=-f$ or
    
  \item $f \in S_2(\Gamma_0(d'p^2))$ for $d'$ a proper divisor of $d$ with $w_pf=f$,
  \end{itemize}
  for which $L(f,\chi) \neq 0$. In \cite[Proposition 5.4]{Angelos} a
  script to perform such computation is given. The problem with this
  approach is that it is almost impossible to compute the space of
  newforms for large values of $N$. In particular, it will definitely
  not work for values of $d$ different from $7$ in the previous list
  (and we did not try to run it for $d=7$ as it will probably not work either).
\end{remark}

\subsection{The case $d=5$:} Let $K= \Q(\sqrt{-5})$ and let
$\id{p}_2 = \langle 2,1+\sqrt{-5}\rangle$ so that $2 =
\id{p}_2^2$. Let $\id{p}_3 = \langle 3,1+\sqrt{-5}\rangle$ so that
$3 = \id{p}_3 \overline{\id{p}_3}$.  By the results of \cite{PT}, the
following holds.
\begin{itemize}
\item The curve $E_{A,B}$ reduction type IV at $\id{p}_2$, conductor
  exponent $2$ and $v_{\id{p}_2}(\Disc(E_{A,B}))=4$. 
  
\item If $3 \mid AB$ then the curve $E_{A,B}$ has reduction type II or
  III at both $\id{p}_3$ and $\overline{\id{p}_3}$.
  
\item If $3 \nmid AB$ and $p \ge 5$ then we can assume that $v_{\id{p}_3}(N) = 1$ while $v_{\overline{\id{p}_3}}(N)=2$.
  
\item At the prime $\id{p}_5 = \sqrt{-5}$ the curve has reduction type IV$^*$, with $v_{\id{p}_5}(\Delta) = 8$. In particular, $e = 3$.
 
\end{itemize}
In all cases, the curve is not the quadratic twist of a curve with
good reduction at primes dividing $2, 3, 5$. Following the notation of
\cite{PT}, $Q_{+-} = \{5\}$ while the other sets are empty. In
particular, $\varepsilon_5$ is the character of order $4$ (and
conductor $5$), while $\varepsilon_2 = \delta_{-1}$, hence
$\varepsilon$ is a character of order $4$ and conductor $20$. By
Theorem 6.3 of loc. cit., we need to compute the spaces
$S_2(\Gamma_0(2^4\cdot 3^2 \cdot 5^2),\varepsilon)$ and
$S_2(\Gamma_0(2^4\cdot 3^3 \cdot 5^2),\varepsilon)$.

\begin{remark}
There are two possible choices for
the character $\chi$ needed to twist the representation
$\rho_{E_{A,B}}$ so that is descends to $G_\Q$, due to the fact that
$\Cl(\Q(\sqrt{-5})) = 2$. Both choices differ by the quadratic twist of
the character attached to the class group (corresponding to
the extension $K(\sqrt{-1})/K$).
\label{remark:d5}
\end{remark}

\noindent $\bullet$ The space
$S_2(\Gamma_0(2^4\cdot 3^2 \cdot 5^2),\varepsilon)$ has $15$ newform
Galois conjugate classes, $7$ of them with CM (which cannot correspond
to solutions). Mazur's trick (as explained in Section
\ref{section:marzurtrick}), proves that all other forms cannot
correspond to solutions if $p>5$ but for four forms, corresponding to
the forms $8, 11, 12$ and $13$ in Magma's order (their coefficient
field is a degree $8$ extension of $\Q$, containing the field of
fourth roots of unity).  By Eichler-Shimura we know that all such
forms have attached a geometric object (providing a $\GL_2$-type
representation), but a priori the dimension of its building block
could be $2$ (see \cite{MR2448720}). A first sanity check (when
possible) is to run Quer's algorithm (implemented in Magma
\cite{MR1484478}) to make sure that the building block has dimension
one, hence there is an elliptic curve attached to our $8$-dimensional
abelian variety. Also, Quer's algorithm gives the field of definition
of the building block. The four forms have inner twists, and have a
building block of dimension one (this can be checked with the
``BrauerClass'' routine and with  the ``DegreeMap'' one in Magma) define
over the quadratic field $\Q(\sqrt{-5})$, hence we are led to compute
such curves.

Before giving equations for the elliptic curves attached to our
modular forms, let us make an important observation. We know that we
cannot have a congruence between two elliptic curves with different
local types (at least if $p \ge 5$), hence we could check whether the
modular forms come as twist of forms of smaller level. The forms
$f_8, f_{12}$ and $f_{13}$ are quadratic twists of forms whose level
have $3$ to the first power while the form $f_{11}$ is a quadratic
twist of a form with good reduction at $3$, hence it cannot come from
a solution! 

We searched for elliptic curves over $\Q(\sqrt{-5})$ with good
reduction outside $\{\id{p}_2,\id{p}_3,\overline{\id{p}_3},\id{p}_5\}$
using Magma's implementation (while running the routine it is useful
to have a bound on the elliptic curve's conductor) with the command
``EllipticCurveWithGoodReductionSearch'', getting $81952$ elliptic
curves. Using a few $a_p$'s of the newforms (using (\ref{eq:BC}) and
twisting by $\chi^{-1}$) we can discard all curves but the ones
related to the previous newforms. Note that since $\Q(\sqrt{-5})$ does
not have class number one, there is no minimal equation. By
Remark~\ref{remark:d5} there are (at least) two elliptic curves
attached to the modular forms $f_i$ (which are quadratic twists of
each other by $K(\sqrt{-1})/K$), and their complex conjugate (which
are isogenous to their twist by $\sqrt{-3}$) hence we will only give
an equation for one of these four curves.

\begin{remark}
  Each curve $E$ that is a candidate to match our modular form has the
  property that it is isogenous to the quadratic twist by $\sqrt{-3}$
  of its Galois conjugate. In particular, they are $\Q$-curves, and
  (after twisting) corresponds to a modular newform in
  $S_2(\Gamma_0(2^4\cdot 3^2 \cdot 5^2),\varepsilon)$. Computing a few
  $a_p$'s allows to certify that the curve is really the building
  block of the newform we started with.
\end{remark}

The curves are the following:
\begin{equation}
  \label{eq:d5-8}
E_8: y^2 = x^3 + (\sqrt{-5}+1)x^2 + (294 \sqrt{-5}+132)x + -1206\sqrt{-5} -6282.
\end{equation}
\begin{equation}
  \label{eq:d5-11}
  E_{11}: y^2 = x^3 + (\sqrt{-5}+1)x^2 + (-186 \sqrt{-5}-168)x + 1334\sqrt{-5} -382.
\end{equation}
The set of curves attached to the form $f_{12}$ is more interesting,
as the curves have an extra isogeny of degree $2$, hence we get eight
different curves (where the isogeny graph of each curve has 4
elements).
\begin{equation}
  \label{eq:d5-12}
  E_{12}:y^2 = x^3 + (-\sqrt{-5}-1)x^2 + (102 \sqrt{-5}-168)x + (810 \sqrt{-5}-162),
\end{equation}
and its two isogenous curve
\[
y^2 = x^3 + (-\sqrt{-5}-1)x^2 + (22 \sqrt{-5}-8)x + (-38 \sqrt{-5}-50).
  \]
  At last, the form $f_{13}$ is related to the elliptic curve
  \begin{equation}
    \label{eq:d5-13}
  E_{13}:  y^2 = x^3 + (\sqrt{-5}+1)x^2 + (414 \sqrt{-5}-2268)x + (-9666 \sqrt{-5}+33318).
  \end{equation}
  Here are some trivial observations: the curve $E_8$ is the quadratic
  twist of a curve with good reduction at $2$, hence it cannot be
  congruent to a curve $E_{A,B}$ (for $p>2$). The curve $E_{11}$ is
  the quadratic twist of a curve with good reduction at $\id{p}_3$
  (after twisting by $3$) hence can also be discarded (we already knew
  this fact). The curve $E_{12}$ is the quadratic twist of a curve
  with good reduction at $\id{p}_5$ hence can also be discarded. So we
  only need to discard the curve $E_{13}$. An easy computation (using
  Sage) shows that the curve $E_{13}$ does not have a $3$-torsion
  point (neither does its $3$-isogenous curve), hence we can use the
  strategy described in Theorem~\ref{thm:Samir} and
  Remark~\ref{remark:siksek}. Concretely,
  $a_{\id{p}_{43}}(E_{13}) = 1$ for both primes in $\Q(\sqrt{-5})$
  dividing $43$. In particular,
  $a_{\id{p}_{43}}(E_{13}) \not \equiv 43+1 \pmod 3$, hence $E_{13}$
  cannot come from a solution if $\ell \ge 4\sqrt{43} = 26.22$...

\vspace{5pt}
  
\noindent $\bullet$ The space
$S_2(\Gamma_0(2^4\cdot 3^3 \cdot 5^2),\varepsilon)$ has $15$ newform
Galois conjugate classes, 6 of which have CM (hence we can discard
them). For the remaining ones, Mazur's trick discard all of them when
$p > 5$ except four forms corresponding to the places $7,8,9$ and $10$
in Magma's output. It is important to remark that twisting by
$\sqrt{-3}$ preserves the space, hence we can actually group the forms
in pairs. The twist of the seventh form is the ninth one, while the
twist of the eighth form is the tenth. Proceeding as before, we compute equations for
their building blocks and obtain
\begin{equation}
  \label{eq:d5II-7}
  E_7:y^2 = x^3 + (-24\sqrt{-5}-60)x + (112 \sqrt{-5}+136),
\end{equation}
and
\begin{equation}
  \label{eq:d5II-8}
  E_8:y^2 + (\sqrt{-5}+1)xy = x^3 + (\sqrt{-5}+1)x^2 + (-45\sqrt{-5}-60)x + (190\sqrt{-5}+100).
\end{equation}
Note that the first curve has good reduction at the prime dividing
$5$, hence cannot be attached to a solution. The second one has
conductor valuation $3$ at both primes dividing $3$, which only occurs
for solutions satisfying $3 \mid AB$.  Once again,
$a_{\id{p}_7}(E_8) = 1$ for both primes dividing $7$, hence this
cannot correspond to a solution if $\ell \ge 4\sqrt{7} = 10.58...$.
%
%
%
%
%
Together with Remark~\ref{thm:ellboundsimp}, we get the following result.

\begin{thm}
  The equation $x^2+5y^5 = z^p$ has no non-trivial solution if
$p \ge 1033$.
\end{thm}
    
\subsection{The case $d=6$:} this case was solved in \cite[Theorem 7.2]{PT}.

\subsection{The case $d=7$:} This case turns to be very interesting,
since the Lebesgue-Nagell equation is hard to solve. In
\cite{MR1259344} (Section 6) Cohn conjectured that all possible
solutions of
\begin{equation}
  \label{eq:Lebesgue-Nagell}
  x^2+7=z^n
\end{equation}
with $n \ge 3$ have $\abs{x} \in \{1, 3, 5, 11, 181\}$. The conjecture
was studied in \cite{MR1980642} and completely solved in
\cite{MR2196761}.
Going back to the general equation $x^2+7y^6=z^p$, following the
notation of \cite{PT}, $Q_{-+} = \{7\}$ and all other ones are empty,
so $\varepsilon$ is the quadratic character of conductor
$21$. According to Theorem 6.3 in \cite{PT} we need to compute the spaces
$S_2(\Gamma_0(2^a\cdot3^b\cdot7^2),\varepsilon)$, where $a \in \{1, 3\}$ and
$b \in \{1, 3\}$.

\vspace{5pt}

\noindent $\bullet$ The space
$S_2(\Gamma_0(2\cdot 3 \cdot 7^2),\varepsilon)$ has two newforms
Galois orbits. The second one can be discarded using Mazur's trick for
$p>7$, but it does not discard the first one. It corresponds to the
elliptic curve:
\begin{equation}
  \label{eq:d=7I}
  E_1:y^2 + xy + \frac{1+\sqrt{-7}}{2}y = x^3 - x^2 + (950\sqrt{-7}-46)x
  + \frac{7285\sqrt{-7}+200449}{2}.  
\end{equation}
In fact, this curve corresponds to the solution
$181^2+7\cdot 1^6=2^{15}$ (notice that this comes from a solution of
~(\ref{eq:Lebesgue-Nagell})). The curve has discriminant
$-1\cdot (\sqrt{-7})^8 \cdot \id{p}_2^{13} \cdot
\overline{\id{p}_2}^{39}\cdot 3^3$ and a point of order $3$ over $K$,
namely, the point $(-5,-2\sqrt{-7}+319 )$. Let
$\id{p}_2 = \langle \frac{1+\sqrt{-7}}{2}\rangle$ and
$\overline{\id{p}_2}$ is Galois conjugate. The curve has
multiplicative reduction at both primes, hence we can apply
Theorem~\ref{thm:classicalsymplectic}. Recall that
$v_{\id{p}_2}(E_{A,B}) \equiv 8 \pmod p$ (and the same is true for
$\overline{\id{p}_2}$), hence both curves are symplectically
isomorphic for $\id{p}_2$ if and only of $\kro{2\cdot 13}{p}=1$ while
for $\overline{\id{p}_2}$ if and only if
$\kro{2 \cdot 13 \cdot 3}{p} = 1$. Clearly both cases cannot occur if
$\kro{3}{p}=-1$.


\vspace{5pt}

\noindent $\bullet$ The space
$S_2(\Gamma_0(2^2\cdot 3\cdot 7^2),\varepsilon)$ has four Galois
orbits of newforms, the first one having CM. The other three forms can
be discarded using Mazur's trick for $p > 7$.

\vspace{5pt}

\noindent $\bullet$ The space
$S_2(\Gamma_0(2^2\cdot 3^3\cdot 7^2),\varepsilon)$ has seven Galois
orbits of newforms. The first three of them have CM (hence can be
discarded), while the other four ones, can be discarded using Mazur's
trick for $p> 7$.

\vspace{5pt}

\noindent $\bullet$ The space
$S_2(\Gamma_0(2\cdot 3^3\cdot 7^2),\varepsilon)$ has six Galois orbits
of newforms. The last three ones can be discarded using Mazur's trick
for $p>7$, but not the first three ones. They correspond to the
following elliptic curves:
\begin{equation}
  \label{eq:d7II-1}
E_1: y^2 + xy + \frac{-1+\sqrt{-7}}{2} y = x^3 -x^2 + \frac{115\sqrt{-7}-91}{2} x + \frac{443\sqrt{-7}+529}{2}.  
\end{equation}
Its discriminant equals
$-1 \cdot (\sqrt{-7})^8\cdot \id{p}_2^9\cdot \overline{\id{p}_2}^3
\cdot 3^3$. It has the $3$ $K$-rational torsion point
$(-5,-2\sqrt{-7}+22)$ (hence cannot be discarded using the method of
Section~\ref{section:3torsion}).
\begin{equation}
  \label{eq:d7II-2}
E_2: y^2 + xy + \frac{-1+\sqrt{-7}}{2} y = x^3 - x^2 + \frac{31\sqrt{-7}-91}{2} x + \frac{121\sqrt{-7}-157}{2}.  
\end{equation}
Its discriminant equals $-1 \cdot (\sqrt{-7})^8\cdot \id{p}_2^6\cdot \overline{\id{p}_2}^2 \cdot 3^3$. It has the $3$ $K$-rational torsion point $(-5,-2\sqrt{-7}+8)$.
\begin{equation}
  \label{eq:d7II-3}
E_3: y^2 + xy + \frac{-1+\sqrt{-7}}{2} y = x^3  -x^2 - \left(\frac{53\sqrt{-7}+91}{2}\right) x - \left(\frac{201\sqrt{-7}+59}{2}\right).  
\end{equation}
Its discriminant equals
$-1 \cdot (\sqrt{-7})^8\cdot \id{p}_2^9\cdot \overline{\id{p}_2}^3
\cdot 3^3$. It has the $3$ $K$-rational torsion point
$(-5,-2\sqrt{-7}-6)$.

Note that in all three cases, we can use again the symplectic argument at both primes dividing $2$, getting that the curve $E_{A,B}$ cannot be symplectic isomorphic to any of the curves if $\kro{3}{p}=-1$.
In particular, we get the following result.
\begin{thm}
  Let $p\ge 337$ be a prime number such that $p \equiv 5,7 \pmod{12}$. Then
  there are no non-trivial solutions of the equation
	\[x^2+7y^6=z^p.\]
\end{thm}

\subsection{The case $d=10$:} Theorem 6.3 of \cite{PT} implies that to
apply the current approach we need to compute the space
$S_2(2^8\cdot3^3\cdot5^2, \varepsilon)$, where $\varepsilon$ is a
character of order $4$ and conductor $5$. Such a computation is
nowadays unfeasible.

\subsection{The case $d=11$:} The prime $2$ is inert in
$\Q(\sqrt{-11})$ while the prime $3$ splits (as in the case
$d=5$). The only non-empty set (following the notation of \cite{PT})
is $Q_{--} = \{11\}$, hence the Nebentypus is trivial. In particular,
we need to compute the spaces
$S_2(\Gamma_0(2^2 \cdot 3^2 \cdot 11^2))$ and
$S_2(\Gamma_0(2^2 \cdot 3^3 \cdot 11^2))$.

\vspace{5pt}

\noindent $\bullet$ The space $S_2(\Gamma_0(2^2\cdot 3^3 \cdot 11^2))$
has $43$ Galois conjugate classes of newforms, $11$ of them having
CM. All other forms in this space can be discarded using Mazur's trick
for $p>13$. In particular it follows that there are no non-trivial primitive solutions if $3\mid AB$.

\vspace{5pt}

\noindent $\bullet$ The space $S_2(\Gamma_0(2^2\cdot 3^2\cdot 11^2))$
has $26$ Galois orbits of newforms, $6$ of them corresponding to forms
with CM. For the other ones, Mazur's trick discards all of them for $p>11$
but two of them whose coefficient field equals $\Q(\sqrt{3})$
(corresponding to the newforms $15$ and $20$). The second one is a
quadratic twist of a form of level $2^2\cdot 11^2$, hence cannot be
congruent to our curve $E_{A,B}$. To discard the form $f_{15}$, we
need to find an equation for the curve (as explained before, we made a
sanity check using Quer's algorithm and computed the dimension of the
building block and its field of definition).  Since the Nebentypus is
trivial, there are two different ways to search for such a curve. One
option is to compute all possible curves with a given set of
ramification using magma; we postpone the second
construction. Discarding the ones that do not match our form, we end
up with the curve
\begin{equation}
  \label{eq:cased11}
E_{15}:y^2 = x^3 + \frac{1+\sqrt{-11}}{2} x^2 + \frac{1907\sqrt{-11}-1615}{2} x - -19479\sqrt{-11}-31012.
\end{equation}

Note that it has a point of order $3$, hence we cannot discard it
using that our solutions are related to elliptic curves with a
$3$-torsion point. In particular, we need to use the symplectic argument
as explained in Section~\ref{section:symplectic}. The discriminant of
$E_{15}$ equals $3^9\cdot (\sqrt{-11})^8 \cdot 2^{20}$. Applying
Theorem~\ref{thm:ramifiedsymplectic} at the prime $\sqrt{-11}$ to both
$E_{15}$ and a curve $E_{A,B}$, we get that they are symplectically
isomorphic because the discriminant valuation of both curves is the
same.

Consider the prime $\id{p}_3 = \langle \frac{1-\sqrt{-11}}{2}\rangle$
of multiplicative reduction of both curves. In particular, we are in
the classical case (i.e. the completion equals $\Q_3$) so we can use
Theorem~\ref{thm:classicalsymplectic}.  The curve $\tilde{E}$ has
discriminant valuation $9$ at $\id{p}_3$. Since the curve $E_{A,B}$
has discriminant valuation $3p-9$ at $\id{p}_3$, they are
symplectically isomorphic if and only if $(-1/p) = 1$. In particular, when $(-1/p)=-1$ we get a contradiction.

\begin{thm}
  The equation $x^2+11y^6 = z^p$ has no non-trivial solution if
  $p \ge 557$ and either one of the following conditions is satisfied:
  \begin{itemize}
  \item Either $x$ or $y$ is divisible by $3$,
    
  \item The prime $p \equiv 3 \pmod 4$.
  \end{itemize}
\end{thm}

\vspace{5pt}

\noindent {\bf A different approach to compute the curve.} We take the
opportunity to explain in detail a method to compute the elliptic
curve as explained in \cite{MR1180252}. We thank Professor John
Cremona for kindly explaining some computational aspects of it. The
Jacobi map (and Eichler-Shimura's relation) allows to given a weight
two newform $f$, construct a lattice $\Lambda_f$ attached to an abelian
variety of $\GL_2$-type (whose complex points correspond to
$\CC^d/\Lambda_f$, where $d = [\Q_f:\Q]$).

In our particular case, let $f = f_{15}$ (to avoid heavy
notation). Recall that $\Q_f = \Q(\sqrt{3})$, a quadratic
extension, hence we can construct a rank $4$ lattice obtained by
integrating the homology against the basis $\{f, f^\sigma\}$, where
$\sigma$ is the non-trivial endomorphism of
$\Gal(\Q(\sqrt{3})/\Q)$. We want to ``split'' the lattice as a sum of
two rank $2$ ones, providing the building block $E_f$ we are searching
for. Following Cremona's notation, it is easy to verify that for all
$\sigma \in \Gal_\Q$,
\[
  \sigma(a_p(f)) = a_p(f)\chi(p),
  \]
  where $\chi$ is the quadratic character corresponding to the
  quadratic extension $\Q(\sqrt{-11})$. By \cite[Theorem
  2]{MR1180252}, the endomorphism algebra of $A_f$ equals the
  quaternion algebra $B=\left(\frac{3,-11}{2}\right) \simeq M_2(\Q)$
  (since $3$ is the norm of $\frac{1+\sqrt{-11}}{2}$). This confirms
  that $A_f$ is isogenous to the product of two elliptic curves over
  $\Q(\sqrt{-11})$.  To split the surface, we need to find a zero
  divisor inside the endomorphism ring, and clearly in $B$, the
  element $1+2i+j$ is such an element (where $i^2=3$, $j^2=-11$ and
  $i j = -j i$). In terms of the lattice, recall that the twisting
  operator $\eta_\chi$ (as an elements of the endomorphism ring) in
  the chosen basis has matrix
  $\left(\begin{smallmatrix} 0 & \sqrt{-11} \\ \sqrt{-11} &
      0\end{smallmatrix}\right)$ (see \cite{MR318162} Section 2). In
  particular, $\eta_\chi^2 = -11$ (the endomorphism given by
  multiplication by $-11$), and
  $f^\sigma = \frac{\eta_\chi f}{\sqrt{-11}}$. In particular,
  $(1+2i+j) f = (1+2\sqrt{3})f +\sqrt{-11}f^\sigma$, so we can
  multiply the lattice (on the left) by the vector
  $[1+2\sqrt{3},\sqrt{-11}]$ to get the complex lattice we are
  searching for. Note that the zero divisor is not unique, a rational
  multiple of it will give another isomorphic (over $\CC$) elliptic
  curve, hence we just compute the $j$-invariant of the elliptic curve
  obtained via this process, and find the twist corresponding to the
  curve that matches $f$.

  Here is how to do it on magma. Just a remark to explain the
  computation: internally Magma works with rational basis, so instead
  of computing the period matrix relative to the pair
  $\{f,f^\sigma\}$, it does so using the basis
  $\{\frac{f+f^\sigma}{2},\frac{f - f^{\sigma}}{2\sqrt{3}}\}$, hence
  we need to multiply by the inverse of such a matrix to get the right
  lattice.

\begin{Verbatim}[fontsize=\small]
  SetDefaultRealFieldPrecision(100);
  M:=ModularSymbols(2^2*3^2*11^2,2);
  S:=NewSubspace(CuspidalSubspace(M));
  new:=NewformDecomposition(S);
  f:=new[15];
  PP:=Periods(f,2000);
\end{Verbatim}
This gives the  matrix of periods

\vspace{5pt}

\resizebox{\linewidth}{!}{%
$\displaystyle
  \begin{pmatrix}
      -0.454169873046895263850024266849 - 6.31088724176809444329382852226\cdot 10^{-30}\text{I} & -0.0484356651074115990113096232058 + 3.62876016401665430489395140030\cdot 10^{-29} \text{I}\\
      0.227084936523447631925012133311 - 0.0246569770083789758390713279419 \text{I} &  0.0242178325537057995056548116133 - 0.129635417833054626720844898122 \text{I}\\
      -1.51694105800301602495812050226 - 0.462877184524300807679748661140 \text{I}& 0.0575601086475070353854284527390 - 0.413563230507542856001606030888 \text{I}\\
      -0.154431438862330233408047701838 +
      0.413563230507542856001606005167 \text{I} &
      0.202867103969741832419357322421 +
      0.154292394841433602559916234538 \text{I}
    \end{pmatrix}
    $}

  \vspace{5pt}
  
To compute the $j$-invariant of our curve (and recognize it as an
algebraic integer), it is better to work in Pari/GP. Here is how the computation finishes

\begin{Verbatim}[fontsize=\footnotesize]
  \p 30
  Periods=[-0.454169873046895263850024266849 - 6.31088724176809444329382852226E-30*I,
  -0.0484356651074115990113096232058 + 3.62876016401665430489395140030E-29*I;
  0.227084936523447631925012133311 - 0.0246569770083789758390713279419*I,
  0.0242178325537057995056548116133 - 0.129635417833054626720844898122*I;
  -1.51694105800301602495812050226 - 0.462877184524300807679748661140*I,
  0.0575601086475070353854284527390 - 0.413563230507542856001606030888*I;
  -0.154431438862330233408047701838 + 0.413563230507542856001606005167*I,
  0.202867103969741832419357322421 + 0.154292394841433602559916234538*I];
  A=[1/2,1/2;1/2/sqrt(3),-1/2/sqrt(3)];
  Candidate=[1+2*sqrt(3),sqrt(-11)]*1/A*Periods~;
  lindep(Candidate)
  %4 = [-7, -15, 1, -4]~
\end{Verbatim}

\noindent This proves that the third element (in $\CC$) is an integral combination of the other three ones.

\begin{Verbatim}[fontsize=\footnotesize]
  lindep([Candidate[1],Candidate[2],Candidate[4]]);
  % 5 = [-1, -3, -2]~
\end{Verbatim}

\noindent Then our lattice is the spanned by the second and fourth elements! We compute the elliptic curve and its j-invariant over the complex numbers.

\begin{Verbatim}[fontsize=\footnotesize]
  W=ellperiods([Candidate[2],Candidate[4]]);
  E=ellinit([0,0,0,-elleisnum(W,4,1)/4,-elleisnum(W,6,1)/4]);
  algdep(E.j,2)
  %8 = 531441*x^2 - 37711872000*x + 1441792000000000
\end{Verbatim}
It is easy to verify that the j-invariant of the
curve~(\ref{eq:cased11}) is a root of the given polynomial.

\subsection{The case $d=13$:} This case can be completely studied
following the method explained in \cite{PT}.  Following the same
notation, we have that the only non-empty set is $Q_{++}=\{13\}$, the
Nebentypus $\varepsilon$ has order $2$ and conductor
$4 \cdot 3 \cdot 19$, while $\chi$ is a character of order $4$. We
have to compute the spaces
$S_2(\Gamma_0(2^4 \cdot 3 \cdot 13^2), \varepsilon)$ and
$S_2(\Gamma_0(2^4 \cdot 3^3 \cdot 13^2), \varepsilon)$.
 
 \vspace{5pt}
 
 \noindent $\bullet$ The space $S_2(\Gamma_0(2^4 \cdot 3 \cdot 13^2)$
 has $29$ conjugacy classes, $9$ of them with CM hence can be
 discarded. The remaining forms can all be eliminated using Mazur's
 trick for $p>13$.

\vspace{5pt}

\noindent $\bullet$ The space
$S_2(\Gamma_0(2^4 \cdot 3^3 \cdot 13^2), \varepsilon)$ has $68$
conjugacy classes, $8$ of them with CM. The remaining forms can all be discarded once again with Mazur's trick for $p>29$.
 
\begin{thm}
  Let $p\ge 3491$ be a prime number. Then there are no non-trivial
  solutions of the equation
	\[x^2+13y^6=z^p.\]
\end{thm}

\subsection{The case $d=14$:} Following the notation of \cite{PT},
$Q_{-+}=\{7\}$, hence the possible solutions are related to forms in
the space $S_2(\Gamma_0(2^8\cdot3^2\cdot7^2), \varepsilon)$ and
$S_2(\Gamma_0(2^8\cdot3^3\cdot7^2), \varepsilon)$, where $\varepsilon$ is the quadratic character of conductor $28$. Nowadays such computation is unfeasible.

\subsection{The case $d=15$:} Let $K=\Q(\sqrt{-15})$. The prime $2$
splits as $2=\id{p}_2\overline{\id{p}_2}$, where
$\id{p}_2=\langle2,\frac{1+\sqrt{-15}}{2}\rangle$ while the prime $3$
is inert in $K$. By Lemmas 5.3 and 5.5 of \cite{PT} the following
holds:
\begin{itemize}
\item $2\nmid AB$ if and only if the curve $E_{A,B}$ has
  multiplicative reduction at $\id{p}_2$ and $\overline{\id{p}_2}$.
\item At the prime $\id{p}_5=\langle 5,\sqrt{-15} \rangle$, the curve
  $E_{A,B}$ has reduction type IV* and $v_{\id{p}_5}(\Delta)=8$. In
  particular, $e=3$.
\end{itemize}
As in the case $d=5$, the unique non-empty set is $Q_{+-}=\{5\}$. By
\cite[Theorem 6.3]{PT} we need to compute the spaces
$S_2(\Gamma_0(2\cdot3^5\cdot5^2),\varepsilon)$ and
$S_2(\Gamma_0(2^2\cdot3^5\cdot5^2),\varepsilon)$ where the
Nebentypus $\varepsilon$ has order $4$ and conductor $3\cdot5$.

\vspace{5pt}

\noindent $\bullet$ The space
$S_2(\Gamma_0(2\cdot3^5\cdot5^2),\varepsilon)$ has 21 conjugacy
classes. Mazur's trick allows to eliminate, for $p>7$, all the
newforms except for the first six ones (in Magma's order) which do not
have CM and have building blocks of dimension $1$. Since the prime $2$
splits in $K$, the character $\chi$ needed to descend our Galois
representation is unramified at $2$, hence the building block has
multiplicative reduction at the primes dividing $2$, and the same
holds for $E_{A,B}$. In particular, we can assume $2\nmid AB$. To
discard the remaining forms we searched for elliptic curves over
$\Q(\sqrt{-15})$ with good reduction outside
$\{\id{p}_2,\overline{\id{p}_2}, \id{p}_3, \id{p}_5\}$ and got
$111264$ elliptic curves. Using a few $a_p$'s of the newforms we
discarded most of the forms, and end with the following equations for
our curves (up to conjugation):
\begin{equation}
  \label{eq:d=15-1}
E_1: y^2 + xy - \frac{1-\sqrt{-15}}{2}y= x^3 - x^2 - \frac{421+23\sqrt{-15}}{2}x - \frac{2185+191\sqrt{-15}}{2},
\end{equation}
\begin{equation}
  \label{eq:d=15-2}
E_2: y^2 + xy - \frac{1-\sqrt{-15}}{2}y= x^3 - x^2 - (2\sqrt{-15}-8)x - \frac{7+7\sqrt{-15}}{2},  
\end{equation}
\begin{equation}
  \label{eq:d=15-3}
E_3: y^2 + xy - \frac{1-\sqrt{-15}}{2}y= x^3 +\frac{1+\sqrt{-15}}{2}x^2 + \frac{(-111+87\sqrt{-15})}{2}x + \frac{65+375\sqrt{-15}}{2},  
\end{equation}
\begin{equation}
  \label{eq:d=15-4}
E_4: y^2 + xy + \frac{1+\sqrt{-15}}{2}y= x^3 - x^2 + -211+137\sqrt{-15}x + \frac{1973+2333\sqrt{-15}}{2},  
\end{equation}
\begin{equation}
  \label{eq:d=15-5}
E_5: y^2 + xy + \frac{1+\sqrt{-15}}{2}y= x^3 + \frac{1-\sqrt{-15}}{2}x^2 - \frac{111+375\sqrt{-15}}{2}x + \frac{-9823+2793\sqrt{-15}}{2},  
\end{equation}
and
\begin{equation}
  \label{eq:d=15-6}
  E_6:y^2 + xy + \frac{1+\sqrt{-15}}{2} y = x^3 - x^2 - (79\sqrt{-15}+211)x - \left(\frac{1339\sqrt{-15}+835}{2}\right).
\end{equation}

The curves $E_2, E_3, E_4$ and $E_5$ have good reduction at $\id{p}_5$
so can be discarded. The curves $E_1$ and $E_6$ do have a $3$-torsion
point (for example $(-11,2-4\sqrt{-15})$ and $(-11,-4\sqrt{-15}-21)$
respectively) hence we need to use the symplectic argument to discard them. Their discriminant valuations are the following:
%
%
\begin{itemize}
\item The curve $E_1$ has minimal discriminat with valuation $6$ at $\id{p}_2$, $2$ at $\overline{\id{p}_2}$, 14 at $\id{p}_3$ and $8$ at $\id{p}_5$.
\item The curve $E_6$ has minimal discriminat with valuation $12$ at
  $\id{p}_2$, $4$ at $\overline{\id{p}_2}$, $14$ at $\id{p}_3$ and $8$
  at $\id{p}_5$.
\end{itemize}
Recall that
$v_{\id{p}_2}(E_{A,B}) \equiv v_{\overline{\id{p}_2}}(E_{A,B}) \equiv
8 -12 \pmod p$ (since $2\nmid AB$ the model is not minimal). If we
apply Theorem~\ref{thm:classicalsymplectic} at both primes dividing
$2$, we get:
\begin{itemize}
\item the method at the prime $\id{p}_2$ gives that $E_1$
  (respectively for $E_6$) and $E_{A,B}$ are symplectically isomorphic
  if and only if $\kro{-6}{p} = 1$ (respectively $\kro{-3}{p}=1$).
  
\item the method at the prime $\overline{\id{p}_2}$ gives that $E_1$
  (respectively $E_6$) and $E_{A,B}$ are symplectically isomorphic if
  and only if $\kro{-2}{p} = 1$ (respectively $\kro{-1}{p}=1$).
\end{itemize}
In particular, if $\kro{3}{p}=-1$ we get a contradiction.

We can do a little better; note that
$v_{\id{p}_5}(E_{A,B}) = 8 = v_{\id{p}_5}(E_1)$ (respectively $E_6$)
hence we can also apply the symplectic argument
(Theorem~\ref{thm:ramifiedsymplectic}) at $\id{p}_5$ getting that the
curves are always symplectically isomorphic. Then we can use this
information to add more primes to the final result. To discard both
curves when $\kro{3}{p}=1$, we need the extra hypothesis
$\kro{-2}{p}=1$ and $\kro{-1}{p}=-1$. In particular, if
$p \equiv 5, 7, 15, 17, 19 \pmod{24}$ then we get a contradiction (we
increased the percentage of primes from $1/2$ to $5/8$).

\vspace{5pt}

\noindent $\bullet$ The space $S_2(\Gamma_0(2^2\cdot3^5\cdot5^2),\varepsilon)$ has 33 conjugacy classes. In this case Mazur's trick is enough to discard all the newforms for $p>11$, except for the first $12$ newforms, which have CM.

\begin{thm}
  The equation $x^2+15y^6=z^p$ has no non-trivial primitive solution
  if $p\ge 743$ and either one of the following conditions is satisfied:
  \begin{itemize}
  \item $2\mid xy$.
  \item $p\equiv 5, 7, 15, 17, 19 \pmod{24}$.
  \end{itemize}
\end{thm}

\subsection{The case $d=17$:} Following \cite{PT} notation, the unique
non-empty set equals $Q_{+-}=\{17\}$. Then the Nebentypus
$\varepsilon$ has order $16$ and conductor $4 \cdot 17$. We need to
discard forms in the spaces
$S_2(\Gamma_0(2^4\cdot 3^2 \cdot 17^2),\varepsilon)$ and
$S_2(\Gamma_0(2^4\cdot 3^3 \cdot 17^2),\varepsilon)$. Although the
dimension of the first space lies in the computable range (close to
the limit), Magma gives an internal error while computing it. The
second space is too large for computations.

\subsection{The case $d=19$:} The unique non-empty set in \cite{PT}
notation is the set $Q_{-+}=\{19\}$. The Nebentypus has order
$2$, conductor $3\cdot19$ and we have to discard newforms in
$S_2(2^2\cdot3\cdot19^2,\varepsilon)$ and
$S_2(2^2\cdot3^3\cdot19^2,\varepsilon)$.

\vspace{5pt}

\noindent $\bullet$ The space
$S_2(\Gamma_0(2^2\cdot 3\cdot19^2),\varepsilon)$ has $10$ conjugacy
classes of newforms, three of them with CM. Mazur's trick allow to
discard all non-CM forms for $p>19$.

\vspace{5pt}

\noindent $\bullet$ The space
$S_2(\Gamma_0(2^2\cdot 3^3\cdot19^2),\varepsilon)$ has $18$ conjugacy
classes of newforms. Mazur's trick allows once again to discard all
the newforms for $p>19$, except for
three of them that have CM.  From the above analysis, the following
holds:
\begin{thm}
  Let $p\ge 1031$ be a prime number. Then there are no non-trivial
  solutions of the equation
	\[x^2+19y^6=z^p.\]
\end{thm}

\bibliographystyle{alpha}
\bibliography{biblio}

\end{document}